\documentclass[12pt]{amsart}
\usepackage{amssymb}
\usepackage{amsmath}
\usepackage{amsthm}
\usepackage{graphicx}
\usepackage{enumerate}
\usepackage{tikz}
\usepackage{tabularx}
\usepackage{tabmac_avi}
\usetikzlibrary{arrows.meta}
\tikzset{>={Latex[width=1.2mm,length=1.7mm]}}

\newtheorem{thm}{Theorem}[section]
\newtheorem{prop}[thm]{Proposition}
\newtheorem{cor}[thm]{Corollary}
\newtheorem{lem}[thm]{Lemma}
\newtheorem{obs}[thm]{Observation}
\newtheorem{conj}[thm]{Conjecture}

\newtheorem{alg}[thm]{Algorithm}
\newtheorem{prob}[thm]{Problem}

\numberwithin{equation}{section}

\newcommand{\sn}{\mathfrak{S}_n}

\newcommand{\mfs}[1]{\mathfrak{S}_{#1}}

\newcommand{\zsn}{\mathbb{Z}[\sn]}
\newcommand{\qsn}{\mathbb{Q}[\sn]}

\newcommand{\qp}[2]{q^{\frac{#1}{#2}}}

\newcommand{\imm}[1]{\mathrm{Imm}_{#1}}

\newcommand{\sumsb}[1]{\sum_{\substack{#1}}}  

\newcommand{\defeq}{:=} 

\newcommand{\spn}{\mathrm{span}}

\newcommand{\sgn}{\mathrm{sgn}}
\newcommand{\triv}{\mathrm{triv}}
\newcommand{\wgt}{\mathrm{wgt}}

\newcommand{\avoidingp}{avoiding the patterns $3412$ and $4231${}}

\newcommand{\tr}{{\negthickspace \top \negthickspace}}

\newcommand{\ttnsp}{\hspace{.5mm}}
\newcommand{\ntnsp}{\negthinspace}
\newcommand{\ntksp}{\negthickspace}
\newcommand{\nTksp}{\negthickspace\negthickspace}

\newcommand{\bp}{\begin{prob}}
\newcommand{\ep}{\end{prob}}

\newcommand{\perm}{\mathrm{per}}
\newcommand{\ctype}{\mathrm{ctype}}

\newcommand{\inc}{\mathrm{inc}}
\newcommand{\permmon}[2]{#1_{1,#2_1} \ntnsp\cdots {#1}_{n,#2_n}}

\newcommand{\ssm}{\smallsetminus}

\newcommand{\slambda}{\mathfrak{S}_\lambda}

\newcommand{\circd}[1]{\raisebox{-9pt}{\textcircled{\raisebox{-.9pt}{#1}}}}
\newcommand{\trspace}[1]{\mathcal{T}_{#1}}
\newcommand{\upparrow}{\big \uparrow \nTksp \phantom{\uparrow}}

\newcommand{\rec}{\mathrm{rec}}
\newcommand{\des}{\mathrm{des}}

\begin{document}
\author{Mark Skandera}

\title{Hook immanantal inequalities for totally nonnegative matrices}

\bibliographystyle{dart}

\date{\today}

\begin{abstract}
  Given a weakly decreasing positive integer sequence
  $\lambda = (\lambda_1,\dotsc,\lambda_\ell)$ summing to $n$,
  let $\chi^\lambda$ denote the irreducible character
  of the symmetric group $\sn$ indexed by $\lambda$.
  This representation has dimension $\chi^\lambda(e)$,
  where $e$ is the identity element of $\sn$.
  Let $\imm{\chi^\lambda}$
  denote the corresponding
  {\em irreducible character immanant},
  the function on $n \times n$ matrices $A = (a_{i,j})$ defined by
\begin{equation*}
  \imm{\chi^\lambda}(A) \defeq \ntksp \sum_{w \in \sn} \ntksp
  \chi^\lambda(w) \ttnsp \permmon aw.
\end{equation*}
Merris conjectured [{\em Linear Multilinear Algebra} {\bf 14} (1983) pp.\ 21--35] 
and Heyfron proved [{\em Linear Multilinear Algebra} {\bf 24} (1988) pp.\ 65--78] 
that irreducible character immanants
indexed by ``hook'' sequences 
$(k, 1, \dotsc, 1)$
satisfy the inequalities
\begin{equation*}
   \perm(A)=\frac{\imm{\chi^n}(A)}{\chi^{n}(e)}\geq \frac{\imm{\chi^{n-1,1}}(A)}{\chi^{n-1,1}(e)}\geq \frac{\imm{\chi^{ n-2,1,1}}(A)}{\chi^{n-2,1,1}(e)}\geq \cdots \geq \frac{\imm{\chi^{1,\dotsc,1}}(A)}{\chi^{1,\dotsc,1}(e)}=\det(A),
\end{equation*} 
whenever $A$ is an $n \times n$ Hermitian positive semidefinite matrix.
We prove that the same inequalities hold
whenever $A$ is an $n \times n$ totally nonnegative matrix.
\end{abstract}

\maketitle

\section{Introduction}\label{s:intro}
A matrix $A \in \mathrm{Mat}_{n \times n}(\mathbb R)$
is called {\em totally nonnegative}
if each minor is nonnegative.
That is, if for all $I, J \subseteq [n] \defeq \{1,\dotsc,n\}$ with $|I| = |J|$,
the submatrix $A_{I,J} \defeq (a_{i,j})_{i \in I, j \in J}$ satisfies
$\det(A_{I,J}) \geq 0$.
A matrix $A \in \mathrm{Mat}_{n \times n}(\mathbb C)$ is called
{\em Hermitian} if it satisfies $A^* = A$ where $*$ denotes conjugate
transpose. Such a matrix is called {\em positive semidefinite}
if we have $y^* A y \geq 0$ for all $y \in \mathbb C^n$.
It is known that this property is equivalent to the condition that
\begin{equation}\label{eq:HPSDminors}
  \det(A_{I,I}) \geq 0 \qquad \text{for all } I \subseteq [n].
  \end{equation}
For $A \in \mathrm{Mat}_{n \times n}(\mathbb R)$, the Hermitian property
reduces to $A^\tr = A$.

Some inequalities satisfied by the entries of
totally nonnegative matrices
are also satisfied by the entries of
Hermitian positive semidefinite matrices, with the latter inequalities
typically being discovered first.
(See \cite[\S 1]{SkanSoskinBJ} for a short exposition.)
A subset of these inequalities concern
expressions
of the form
\begin{equation}\label{eq:immfrac}
  \frac{\imm{\theta}(A)}{\theta(e)},
\end{equation}
where $\theta: \sn \rightarrow \mathbb Z$ is an $\sn$-character,
$e$ is the identity element of $\sn$, and $\imm{\theta}(A)$ is the
{\em $\theta$-immanant}
defined by
\begin{equation}\label{eq:immdef}
  \imm{\theta}(A) = \sum_{w \in \sn} \ntnsp \theta(w) \permmon aw.
\end{equation}

Since $\theta(e)$ is the dimension of the representation
having character $\theta$, the ratio (\ref{eq:immfrac})
is clearly real when $A$ is totally nonnegative.
To see that it is also real when $A$ is Hermitian positive semidefinite,
let $\trspace n$ be the
real vector space
spanned by all $\sn$-characters,
and recall that $\dim (\trspace n)$ 
equals the number of
{\em integer partitions} of $n$,
weakly decreasing positive integer sequences
$\lambda$
summing to $n$.
Let
$\lambda \vdash n$ denote that $\lambda$ is a partition of $n$,
and consider
the irreducible character basis
$\{ \chi^\lambda \,|\, \lambda \vdash n \}$ and
induced sign character basis
$\{ \epsilon^\lambda \,|\, \lambda \vdash n \}$ of $\trspace n$.
Arbitrary $\sn$-characters $\theta$ satisfy
\begin{equation}\label{eq:thetainepsilonspan}
  \theta \in \spn_{\mathbb N}\{\chi^\lambda \,|\, \lambda \vdash n \}
  \subseteq \spn_{\mathbb Z}\{\epsilon^\lambda \,|\, \lambda \vdash n \},
  \end{equation}
with
\begin{equation}\label{eq:epsilonlincombo}
  \theta = \sum_{\lambda \vdash n} b_\lambda \epsilon^\lambda
  \quad
  \Longleftrightarrow
  \quad
  \imm \theta(A) = \sum_{\lambda \vdash n} b_\lambda \imm{\epsilon^\lambda}(A),
  \end{equation}
and the
Littlewood--Merris--Watkins
identity~\cite{LittlewoodTGC}, \cite{MerWatIneq} asserts that
\begin{equation}\label{eq:lmw}
  \imm{\epsilon^\lambda}(A) = \sum \det(A_{I_1,I_1}) \cdots \det(A_{I_\ell,I_\ell}),
\end{equation}
where $\ell = \ell(\lambda)$ is the number of nonzero components of $\lambda$,
and the sum is over all sequences $(I_1,\dotsc,I_\ell)$ of disjoint
subsets of $[n]$ satisfying $|I_j|= \lambda_j$.
This number is real by (\ref{eq:HPSDminors}), and therefore
by (\ref{eq:thetainepsilonspan}) -- (\ref{eq:epsilonlincombo})
the number (\ref{eq:immfrac}) is real as well.




The study of
inequalities 
for the expressions (\ref{eq:immfrac}) evaluated at Hermitian
positive semidefinite matrices was originally motivated by work of 
Hadamard and Schur,
and was later remotivated by generalizations of Marcus and Lieb.
(See \cite[\S 1]{WLiebPDC}.)
The study of inequalities for the expressions evaluated at totally nonnegative
matrices is motivated by Lusztig's work on quantum groups and their
connection to total nonnegativity~\cite{LusztigTP}.
(See also \cite[\S 1]{SkanNNDCB}.)


Barrett and Johnson~\cite{BJMajor} showed that
expressions (\ref{eq:immfrac})
with induced sign characters
satisfy
\begin{equation}\label{eq:indsgnimmineq}
  \frac{\imm{\epsilon^\lambda}(A)}{\epsilon^\lambda(e)}
  \geq
  \frac{\imm{\epsilon^\mu}(A)}{\epsilon^\mu(e)}
\end{equation}
for all real positive semidefinite matrices
if and only if {\em $\lambda$ is majorized by $\mu$},
i.e., if and only if we have
$\lambda_1 + \cdots + \lambda_i \leq \mu_1 + \cdots + \mu_i$ for all $i$.
The author and Soskin~\cite{SkanSoskinBJ}
showed that the inequalities (\ref{eq:indsgnimmineq})
hold for all $n \times n$ totally nonnegative matrices
under the same conditions on $\lambda$ and $\mu$.
The problem of characterizing the inequalities (\ref{eq:indsgnimmineq})
which hold for all Hermitian positive semidefinite matrices remains
open.

Borcea--Br\"and\'en proved inequalities involving products of expressions
(\ref{eq:immfrac}) when $A$ is Hermitian positive semidefinite.
In particular for
$k = 2,\dotsc, n-1$ we have~\cite[Cor.\,3.1(b)]{BBApps}
  \begin{equation}\label{eq:borcea1}
    \left(\frac{\imm{\epsilon^{k,n-k}}(A)}{\epsilon^{k,n-k}(e)}\right)^{\ntksp2}
    \geq
    \left(\frac{\imm{\epsilon^{k+1,n-k-1}}(A)}{\epsilon^{k+1,n-k-1}(e)}\right)
    \left(\frac{\imm{\epsilon^{k-1,n-k+1}}(A)}{\epsilon^{k-1,n-k+1}(e)}\right)\ntnsp;
  \end{equation}
  for
  $k = 1,\dotsc, n-1$ and $A$ positive definite
  we have~\cite[Cor.\,3.1(c)]{BBApps}
  \begin{equation}\label{eq:borcea2}
    \left(\frac{\imm{\epsilon^{k+1,n-k-1}}(A)}{\epsilon^{k+1,n-k-1}(e)}\right)^{\ntksp k}
    \ntksp \det(A)
    =
    \left(\frac{\imm{\epsilon^{k+1,n-k-1}}(A)}{\epsilon^{k+1,n-k-1}(e)}\right)^{\ntksp k}
    \frac{\imm{\epsilon^n}(A)}{\epsilon^n(e)}
    \geq
    \left(\frac{\imm{\epsilon^{k,n-k}}(A)}{\epsilon^{k,n-k}(e)}\right)^{\ntksp k+1}\nTksp\nTksp\ntnsp.
  \end{equation}
  Inequalities (\ref{eq:borcea1}) -- (\ref{eq:borcea2})
  are not known to hold for (invertible) totally nonnegative matrices.

Schur~\cite{SchurUberendliche} proved that for each $\sn$-character $\theta$,
the inequality
\begin{equation}\label{eq:schur}
  \frac{\imm{\chi^{1,\dotsc, 1}}(A)}{\chi^{1, \dots, 1}(e)} = \det(A)
  \leq \frac{\imm \theta(A)}{\theta(e)}
\end{equation}
holds for all Hermitian positive semidefinite matrices.
Stembridge~\cite[Cor.\,3.4]{StemImm} showed that it holds
for totally nonnegative matrices as well.
Johnson showed that for each $\sn$-character $\theta$, the
permanental analog 
\begin{equation}\label{eq:liebdominance}
  \frac{\imm{\chi^n}(A)}{\chi^n(e)} = \perm(A)
  \geq \frac{\imm \theta(A)}{\theta(e)}
\end{equation}
of (\ref{eq:schur}) holds for totally nonnegative matrices
(unpublished; see \cite[p.\,1088]{StemConj}).
Lieb~\cite{LiebPerm} conjectured the same for Hermitian positive semidefinite
matrices.  This statement, known as the {\em permanental dominance conjecture}
is still open.

The problems of characterizing inequalities of the form
\begin{equation}\label{eq:irrimmineq}
  \frac{\imm{\chi^\lambda}(A)}{\chi^\lambda(e)}
  \geq
  \frac{\imm{\chi^\mu}(A)}{\chi^\mu(e)}
\end{equation}
for Hermitian positive semidefinite matrices and totally nonnegative matrices
are open,
with preliminary work of James~\cite[Appendix]{JamesImm}
and Stembridge~\cite[\S 3]{StemConj},
not suggesting any simple criterion with which to compare
$\lambda$, $\mu$.
On the other hand,
Merris conjectured~\cite[\S 4]{MerrisSingleHook}
and Heyfron proved~\cite[Thm.\,1]{HeyfronImmDom}
a characterization of
the subset of inequalities (\ref{eq:irrimmineq})
which hold for Hermitian positive semidefinite matrices
when $\lambda$ and $\mu$ are {\em hook partitions}, i.e., have
the form
\begin{equation*}
  k1^{n-k} \defeq
  (k, \; \underbrace{\ntksp 1, \dotsc, 1 \ntksp }_{n-k}\;).
  \end{equation*}
\begin{thm}\label{t:heyfron}
  For $A$ an $n \times n$ Hermitian positive semidefinite matrix we have
\begin{equation*}
   \perm(A)=\frac{\imm{\chi^n}(A)}{\chi^{n}(e)}\geq \frac{\imm{\chi^{n-1,1}}(A)}{\chi^{n-1,1}(e)}\geq \frac{\imm{\chi^{ n-2,1,1}}(A)}{\chi^{n-2,1,1}(e)}\geq \cdots \geq \frac{\imm{\chi^{1,\dotsc,1}}(A)}{\chi^{1,\dotsc,1}(e)}=\det(A).
\end{equation*} 
\end{thm}
We will prove that the same inequalities hold
whenever $A$ is totally nonnegative.
In Section~\ref{s:tracesf} we review symmetric functions and traces.
In Sections~\ref{s:chrom} -- \ref{s:tnn}
we discuss chromatic symmetric functions, posets,
and their applications to total nonnegativity.
This leads to two proofs of our main theorem in Section~\ref{s:main}.


\section{Symmetric functions and $\sn$-traces}\label{s:tracesf}
Inside of $\trspace n$, the $\mathbb Z$-module
$\spn_{\mathbb Z}\{ \chi^\lambda \,|\, \lambda \vdash n \}$
of {\em virtual characters}
is isomorphic to the $\mathbb Z$-module $\Lambda_n$
of homogeneous symmetric functions of degree $n$.
Six standard bases of $\Lambda_n$ consist of the
monomial $\{m_\lambda \,|\, \lambda \vdash n\}$,
elementary $\{e_\lambda \,|\, \lambda \vdash n\}$,
(complete) homogenous $\{h_\lambda \,|\, \lambda \vdash n\}$,
power sum $\{p_\lambda \,|\, \lambda \vdash n\}$,
Schur $\{s_\lambda \,|\, \lambda \vdash n\}$, and
forgotten $\{f_\lambda \,|\, \lambda \vdash n\}$ symmetric functions.
(See, e.g., \cite[Ch.\,7]{StanEC2} for definitions.)
An involutive automorphism $\omega: \Lambda \rightarrow \Lambda$ defined
by $\omega(e_k) = h_k$ for all $k$ acts on these bases of $\Lambda_n$ by
\begin{equation*}
  \omega(s_\lambda) = s_{\lambda^\tr}, \qquad
  \omega(m_\lambda) = f_\lambda, \qquad
  \omega(e_\lambda) = h_\lambda, \qquad
  \omega(p_\lambda) = (-1)^{n-\ell(\lambda)}p_\lambda,
\end{equation*}
where we define the {\em transpose} partition
$\lambda^\tr = (\lambda^\tr_1, \dotsc, \lambda^\tr_{\lambda_1})$ of $\lambda$
by
\begin{equation*}
  \lambda^\tr_i = \# \{ j \,|\, \lambda_j \geq i \}.
\end{equation*}

Linking these two $\mathbb Z$-modules is the Frobenius isomorphism
\begin{equation}\label{eq:Frob}
  \begin{aligned}
  \mathrm{Frob}: \spn_{\mathbb Z}\{ \chi^\lambda \,|\, \lambda \vdash n \}
  &\rightarrow \Lambda_n\\
  \theta &\mapsto \frac 1{n!} \sum_{w \in \sn} \ntnsp \theta(w) p_{\ctype(w)},
  \end{aligned}
\end{equation}
where $\ctype(w)$ is the cycle type of $w$.
This maps the bases of
irreducible characters,
induced sign characters,
and induced trivial characters of $\sn$
to the Schur, elementary, and homogeneous bases of $\Lambda_n$, respectively,
\begin{equation*}
  \chi^\lambda \mapsto s_\lambda,
  \qquad
  \epsilon^\lambda \defeq \sgn \ntnsp \upparrow_{\slambda}^{\sn} \mapsto e_\lambda,
  \qquad
  \eta^\lambda \defeq \triv \ntnsp \upparrow_{\slambda}^{\sn} \mapsto h_\lambda.  
\end{equation*}
Here $\slambda$ is the Young subgroup of $\sn$
indexed by $\lambda$.  (See, e.g., \cite{Sag}.)
The power sum,
monomial,
and forgotten bases of $\Lambda_n$ correspond
to
bases of $\trspace n$
which are not characters.
We call these the
{\em power sum}
$\{\psi^\lambda \,|\, \lambda \vdash n \}$,
{\em monomial}
$\{ \phi^\lambda \,|\, \lambda \vdash n \}$,
and {\em forgotten}
$\{ \gamma^\lambda \,|\, \lambda \vdash n \}$
traces of $\sn$, respectively.
These
are
the virtual character bases related to the irreducible character bases
by the same matrices of character evaluations and
Kostka numbers that relate
power sum, monomial, and forgotten symmetric functions to Schur functions,
\begin{equation}\label{eq:mforgotten}
  \begin{alignedat}{3}
    p_\lambda &= \sum_\mu \chi^\mu(\lambda) s_\mu, &\qquad
    s_\lambda &= \sum_\mu K_{\lambda,\mu} m_\mu, &\qquad
    s_\lambda &= \sum_\mu K_{\lambda^\tr,\mu} f_\mu,\\
    \psi^\lambda &= \sum_\mu \chi^\mu(\lambda) \chi^\mu, &\qquad
    \chi^\lambda &= \sum_\mu K_{\lambda,\mu} \phi^\mu, &\qquad
    \chi^\lambda &= \sum_\mu K_{\lambda^\tr,\mu} \gamma^\mu,
  \end{alignedat}
\end{equation}
where $\chi^\mu(\lambda) \defeq \chi^\mu(w)$
for any $w \in \sn$ having cycle type $\lambda$.
 The power sum traces of $\trspace n$
also have the natural definition
\begin{equation}\label{eq:psidef}
  \psi^\lambda(w) \defeq \begin{cases}
    z_\lambda &\text{if $\ctype(w) = \lambda$},\\
    0 &\text{otherwise},
  \end{cases}
\end{equation}
where $z_\lambda = \lambda_1 \cdots \lambda_\ell \alpha_1! \cdots \alpha_n!$
and $\alpha_i$ is the number of components
of $\lambda$ equal to $i$.

The character evaluations $\chi^\mu(\lambda)$ can be performed using
the {\em Murnaghan--Nakayama rule}, 
or in the special case that $\lambda = 1^n$, the
{\em hooklength formula}. (See, e.g., \cite[\S 7.17, \S 7.21]{StanEC2}.)
Alternatively, for
$\mu  = (\mu_1,\dotsc,\mu_r) \vdash n$,
the
number
$\chi^\mu(1^n) = \chi^\mu(e)$
counts
standard Young tableaux of shape $\mu$.
A {\em Young diagram of shape $\mu$} consists
of $r$ left-justified rows of boxes with $\mu_i$ boxes in row $i$,
and a {\em standard Young tableau of shape $\mu$}
is a filling of this 
with $1, \dotsc, n$,
so that entries strictly increase in rows and columns.
For example, in $\mfs 4$ we have
$\chi^{31}(e) = 3$:
\begin{equation*}
  \tableau[scY]{2|1,3,4}\,,
  \qquad \tableau[scY]{3|1,2,4}\,,
  \qquad \tableau[scY]{4|1,2,3}\,.
  \end{equation*}
When $\mu$ is a hook partition, this leads to the following simple expression.
\begin{obs}\label{o:hookchie}
  We have $\chi^{k1^{n-k}}(e) = \tbinom{n-1}{k-1}$.
\end{obs}
\begin{proof}
  Filling a Young diagram of shape $k1^{n-k}$ with $\{1,\dotsc,n\}$,
  we place $1$ in position $(1,1)$, we choose $k-1$ letters from
  $\{2, \dotsc, n\}$
  to follow it in row $1$ in increasing order,
  and place the remaining letters in column $1$ in increasing order.
\end{proof}


The Kostka number $K_{\lambda,\mu}$ equals the number of
{\em semistandard Young tableaux of shape
  $\lambda$ and content $\mu$}
i.e., Young diagrams of shape $\lambda$
filled with $\mu_1$ ones, $\mu_2$ twos, etc.,
so that entries
strictly increase in columns and weakly increase in rows.
For example, when $\lambda = 411$
we have $K_{411,2211} = K_{411,3111} = 3$:
\begin{equation*}
  \tableau[scY]{4|3|1,1,2,2}\,,
  \quad
  \tableau[scY]{4|2|1,1,2,3}\,,
  \quad
  \tableau[scY]{3|2|1,1,2,4}\,,
  \quad
  \tableau[scY]{4|3|1,1,1,2}\,,
  \quad
  \tableau[scY]{4|2|1,1,1,3}\,,
  \quad
  \tableau[scY]{3|2|1,1,1,4}\,.
  \quad
\end{equation*}
When $\lambda$ is a hook partition $k1^{n-k}$,
we have a simple formula for $K_{\lambda,\mu}$ which
depends only upon $k$ and 
the number $\ell(\mu)$ of components of $\mu$.
This leads to a simple expansion of the corresponding irreducible
characters in terms of monomial traces.

\begin{prop}\label{p:hookKostka}
  For each hook partition $k1^{n-k} \vdash n$
  and each arbitrary partition $\mu \vdash n$,
  we have
  \begin{equation}\label{eq:hookKostka}
    K_{k1^{n-k}\ntnsp,\;\mu} =
    \binom{\ell(\mu)-1}{n-k},
  \end{equation}
  and therefore
  \begin{equation}\label{eq:hookexpansion}
    \chi^{k1^{n-k}} = \ntksp \sum_{\ell = n-k+1}^n \ntksp \binom{\ell-1}{n-k}
    \ntksp \sumsb{\mu \vdash n\\\ell(\mu) = \ell} \ntksp \phi^\mu.
\end{equation}
\end{prop}
\begin{proof}
  To see (\ref{eq:hookKostka}), observe that for $k = 1,\dotsc,n$,
  each semistandard Young tableau $U$ of shape $k1^{n-k}$ and content $\mu$
  has the entry $1$ in position $(1,1)$.
  The remaining $n-k$ entries of column $1$ must be distinct,
  must be chosen from the set $\{2,\dotsc,\ell = \ell(\mu)\}$, and
  must appear in increasing order.
  There are $\tbinom{\ell-1}{n-k}$ ways to choose these.
  The remaining $k-1$ entries of row $1$ are then uniquely determined by
  completing the multiset $1^{\mu_1} \cdots \ell^{\mu_\ell}$.  
  These appear
  in row $1$ in weakly increasing order.  Now apply (\ref{eq:mforgotten})
  to obtain (\ref{eq:hookexpansion}).
\end{proof}


It can be useful to record trace evaluations
in a symmetric generating function.
In particular, for $D \in \qsn$,
we record induced sign character evaluations by defining
\begin{equation}\label{eq:ydef}
  Y(D) \defeq \sum_{\lambda \vdash n} \epsilon^\lambda(D) m_\lambda
  \in \mathbb Q \otimes \Lambda_n.
\end{equation}
This symmetric generating function in fact gives all
standard trace evaluations $\theta(D)$.
(See \cite[Prop.\,2.1]{SkanCCS}.)
\begin{prop}\label{p:Yexpand}
The symmetric function $Y(D)$ is equal to
  \begin{equation*}
      \sum_{\lambda \vdash n} \eta^\lambda(D)f_\lambda
    = \sum_{\lambda \vdash n}\frac{(-1)^{n-\ell(\lambda)}\psi^\lambda(D)}{z_\lambda}p_\lambda
    = \sum_{\lambda \vdash n} \chi^{\lambda^\tr}(D)s_\lambda
    = \sum_{\lambda \vdash n} \phi^\lambda(D)e_\lambda
    = \sum_{\lambda \vdash n} \gamma^\lambda(D)h_\lambda.
  \end{equation*}
\end{prop}
\noindent
Furthermore, every element of $\Lambda_n$ is a special case of this.
(See \cite[Prop.\,2.3]{SkanCCS}.)
\begin{prop}\label{p:everysymmfn}
   Every symmetric function in $\Lambda_n$
  has the form $Y(D)$ for some element $D \in \qsn$.
\end{prop}
    
\section{Chromatic symmetric functions and posets}\label{s:chrom}

Certain symmetric functions which Stanley~\cite{StanSymm}
associated to graphs are related by Proposition~\ref{p:everysymmfn}
to a subset of the Kazhdan--Lusztig basis $\{ C'_w(1) \,|\, w \in \sn \}$
of $\zsn$ defined in \cite{KLRepCH}.
Specifically, this subset is
$\{ C'_w(1) \,|\, w \in \sn \text{ avoids the pattern } 312 \}$~\cite[Thm.\,7.1]{CHSSkanEKL}.
We make this relationship precise in Proposition~\ref{p:uioto312avoid} and then
state some consequences.

Define a {\em proper coloring} of a (simple undirected) graph $G = (V,E)$
to be an assignment $\kappa: V \rightarrow \{1,2,\dotsc, \}$
of colors (positive integers) to $V$
such that adjacent vertices have different colors.
For $G$ on $|V| = n$ vertices and any partition
$\lambda = (\lambda_1,\dotsc,\lambda_\ell) \vdash n$,
say that a coloring $\kappa$ of $G$ has {\em type}
$\lambda$ if $\lambda_i$ vertices have color $i$ for $i = 1,\dotsc, \ell$.
Let $c(G,\lambda)$ be the number of proper colorings of $G$ of type
$\lambda$.
Define the {\em chromatic symmetric function}
of $G$ to be
\begin{equation}\label{eq:chrom}
  X_G \defeq
  \sum_{\lambda \vdash n} c(G,\lambda) m_\lambda.
\end{equation}
By Proposition~\ref{p:everysymmfn},
we see that for each graph $G$ on $n$
vertices, there exists an element $D \in \qsn$ such that $X_G = Y(D)$.
While $G$ does not uniquely determine such an element $D$,
it does uniquely determine all trace evaluations $\theta(D)$.
(See \cite[Obs.\,3.1]{SkanCCS}.)
\begin{obs}\label{o:yyx}
  Let $G$ be a graph on $n$ vertices
  and let $D \in \qsn$ satisfy $Y(D) = X_G$.
  Then for each trace
  $\theta  = \sum_{\lambda \vdash n} a_\lambda \epsilon^\lambda \in \trspace n$,
  we have
  $\theta(D) = \sum_{\lambda \vdash n} a_\lambda c(G,\lambda)$.
\end{obs}
For every trace
$\theta \in \trspace n$, Proposition~\ref{p:everysymmfn} and
Observation~\ref{o:yyx} allow us to define
\begin{equation}\label{eq:thetaG}
  \theta(G) \defeq \theta(D),
\end{equation}
where $D$ is any element in $\qsn$ satisfying $Y(D) = X_G$.
By Proposition~\ref{p:Yexpand}, we have that
$c(G,\lambda) = \epsilon^\lambda(D)$ and furthermore that
$X_G$ is equal to
  \begin{equation}\label{eq:sumXG}
      \sum_{\lambda \vdash n} \eta^\lambda(G)f_\lambda
    = \sum_{\lambda \vdash n}\frac{(-1)^{n-\ell(\lambda)}\psi^\lambda(G)}{z_\lambda}p_\lambda
    = \sum_{\lambda \vdash n} \chi^{\lambda^\tr}(G)s_\lambda
    = \sum_{\lambda \vdash n} \phi^\lambda(G)e_\lambda
    = \sum_{\lambda \vdash n} \gamma^\lambda(G)h_\lambda.
  \end{equation}


Some conditions on graphs $G$ and traces $\theta$ imply
the numbers $\theta(G)$
to be positive and possibly to have nice algebraic and
combinatorial interpretations.
This is especially true
when $G$ is the
incomparability graph of a poset which is a unit interval order.
(See, e.g., \cite{CHSSkanEKL}, \cite{SWachsChromQF}, \cite{StanSymm}.)
Given a poset $P$, define its {\em incomparability graph} $\inc(P)$
to be the graph having a vertex for each element of $P$ and
an edge $\{i,j\}$ for each incomparable pair of elements of $P$.
Call the poset a {\em unit interval order} if it has no induced subposet
isomorphic to the posets
\begin{equation*}
  (\mathbf3 + \mathbf1) =\,
  \begin{tikzpicture}[scale=.7,baseline=-5]
\draw[fill] (0,1) circle (1mm); 
\draw[fill] (0,0) circle (1mm); 
\draw[fill] (0,-1) circle (1mm); 
\draw[fill] (.8,0) circle (1mm); 
\draw[-,thick] (0,1) -- (0,0);
\draw[-,thick] (0,0) -- (0,-1);
\end{tikzpicture}\, ,
  \qquad
  (\mathbf2 + \mathbf2) =\,
\begin{tikzpicture}[scale=.7,baseline=-5]
\draw[fill] (0,.5) circle (1mm); 
\draw[fill] (0,-.5) circle (1mm); 
\draw[fill] (.8,.5) circle (1mm); 
\draw[fill] (.8,-.5) circle (1mm); 
\draw[-,thick] (0,.5) -- (0,-.5);
\draw[-,thick] (.8,.5) -- (.8,-.5);
\end{tikzpicture}\,.
\end{equation*}
For example,
the
following
unit interval order
and
graph
\begin{equation}\label{eq:n+1}
  P = \ntnsp
\begin{tikzpicture}[scale=.7,baseline=-5]
\draw[fill] (0,.5) circle (1mm); \node at (0,1) {$4$};
\draw[fill] (0,-.5) circle (1mm); \node at (0,-1) {$1$};
\draw[fill] (.8,.5) circle (1mm); \node at (.8,1) {$5$};
\draw[fill] (.8,-.5) circle (1mm); \node at (.8,-1) {$2$};
\draw[fill] (1.6,0) circle (1mm); \node at (1.6,-.5) {$3$};
\draw[-,thick] (0,.5) -- (0,-.5);
\draw[-,thick] (.8,.5) -- (0,-.5);
\draw[-,thick] (.8,.5) -- (.8,-.5);
\end{tikzpicture}\,,
\qquad \qquad
  G = \ntnsp
\begin{tikzpicture}[scale=.7,baseline=-5]
\draw[fill] (-1.2,.5) circle (1mm); \node at (-1.2,1) {$1$};
\draw[fill] (-.4,.5) circle (1mm); \node at (-.4,1) {$2$};
\draw[fill] (.4,.5) circle (1mm); \node at (.4,1) {$4$};
\draw[fill] (0,-.5) circle (1mm); \node at (0,-1) {$3$};
\draw[fill] (1.2,.5) circle (1mm); \node at (1.2,1) {$5$};
\draw[-,thick] (-1.2,.5) -- (0,-.5);
\draw[-,thick] (-.4,.5) -- (0,-.5);
\draw[-,thick] (.4,.5) -- (0,-.5);
\draw[-,thick] (1.2,.5) -- (0,-.5);
\draw[-,thick] (-1.2,.5) -- (-.4,.5);
\draw[-,thick] (-.4,.5) -- (.4,.5);
\draw[-,thick] (.4,.5) -- (1.2,.5);
\end{tikzpicture}\,
\end{equation}
satisfy $\inc(P) = G$.

Given an $n$-element unit interval order $P$,
it is easy to explicitly describe
an element $D = D(P) \in \zsn$
satisfying $Y(D) = X_{\inc(P)}$~\cite[\S 3]{SkanCCS}.  
\begin{alg}\label{a:PtoC}
  Given unit interval order $P$, do
  \begin{enumerate}
  \item For each element $y \in P$, compute
$\beta(y) \defeq
  \# \{ x \in P \,|\, x \leq_P y \} - \# \{ z \in P \,|\, z \geq_P y \}.$
  \item Label the poset elements $1, \dotsc, n$
so that we have
  $\beta(1) \leq \cdots \leq \beta(n)$.
\item Define $w = w(P) = w_1 \cdots w_n$ by
  $w_j = \max ( \{ i \,|\, i \not >_P j \} \ssm \{ w_1, \dotsc, w_{j-1} \} )$.
\item Define $D = \sum_{v \leq w} v$, where $\leq$ is the Bruhat order on $\sn$.
\end{enumerate}
\end{alg}
\noindent
The element $D$ produced by Algorithm~\ref{a:PtoC} is usually written
$C'_w(1)$ and is called the {\em Kazhdan--Lusztig basis element}
indexed by $w$~\cite{KLRepCH}.  (See \cite{BBCoxeter} for more
information on this basis and the Bruhat order.)
The map $P \mapsto w(P)$ defined by Steps 1--3 of Algorithm~\ref{a:PtoC}
is a bijection from $n$-element unit interval orders to the
$\tfrac{1}{n+1}\tbinom{2n}n$ $312$-avoiding permutations in $\sn$,
and gives us the following result~\cite[Cor.\,7.5]{CHSSkanEKL}.
(See \cite{BilleyLak} for more information on pattern avoidance.)
\begin{prop}\label{p:uioto312avoid}
  Let $P$ be an $n$-element unit interval order and $w = w(P)$ be the
  corresponding $312$-avoiding permutation in $\sn$.  Then we have
  $Y(C'_w(1)) = X_{\inc(P)}$.
\end{prop}

For example, elements of the poset $P$ in (\ref{eq:n+1}) are already labeled
as in Step 2 of Algorithm~\ref{a:PtoC}:
  $(\beta(1), \beta(2), \beta(3), \beta(4), \beta(5)) = (-2, -1, 0, 1, 2)$.
Thus we compute
\begin{equation*}
  w(P) = 34521,
  \qquad D = C'_{34521}(1) = \nTksp \sum_{v \leq 34521} \nTksp v,
  \end{equation*}
and we have $Y(C'_{34521}(1)) = X_{\inc(P)}$.

The labeling in Step 2 of
Algorithm~\ref{a:PtoC}~\cite[p.\,33]{Fish}, \cite[\S 8.2]{Trott}
also associates a totally nonnegative matrix to a
unit interval order $P$.
Define the {\em antiadjacency matrix} of a labeled poset $P$ to be the matrix
$A = (a_{i,j})$ with entries
\begin{equation}\label{eq:antiadj}
  a_{i,j} = \begin{cases}
    0 &\text{if $i <_P j$},\\
    1 &\text{otherwise}.
  \end{cases}
\end{equation}
\begin{prop}
  For $P$ a unit interval order labeled as in
  Algorithm~\ref{a:PtoC},
  the antiadjacency matrix $A = A(P)$ is totally nonnegative.
\end{prop}
\begin{proof}
  Entries of $A$ equal to $0$ form a Ferrers diagram in the upper right
  of the matrix.  Thus each submatrix of $A$ not having a row or column
  of zeros must be unitriangular
  or must have repeated rows or columns.
\end{proof}

Combinatorial interpretations of numbers $\theta(\inc(P))$ often
involve generalizations of Young tableaux
called $P$-tableaux, and statistics on these.
Define a {\em $P$-tableau of shape $\lambda \vdash |P|$} to be
a filling of a
Young diagram of shape $\lambda$ with the
elements of $P$, one per box.
For example we have the poset $P$ and $P$-tableaux
\begin{equation}\label{eq:tableauxposet}
  P =
\begin{tikzpicture}[scale=.7,baseline=-5]
\draw[fill] (0,1) circle (1mm); \node at (-.5,1) {$5$};
\draw[fill] (0,0) circle (1mm); \node at (-.5,0) {$3$};
\draw[fill] (0,-1) circle (1mm); \node at (-.5,-1) {$1$};
\draw[fill] (.8,.5) circle (1mm); \node at (1.3,.5) {$4$};
\draw[fill] (.8,-.5) circle (1mm); \node at (1.3,-.5) {$2$};
\draw[-,thick] (0,1) -- (0,0);
\draw[-,thick] (0,0) -- (0,-1);
\draw[-,thick] (0,1) -- (.8,-.5);
\draw[-,thick] (0,-1) -- (.8,.5);
\draw[-,thick] (.8,.5) -- (.8,-.5);
\end{tikzpicture},
\qquad \qquad%
\begin{gathered}
  T \ntnsp= {\tableau[scY]{5|2,4|1,3}}\,, \quad 
  U = {\tableau[scY]{5,4|1,2,3}}\,, \quad
  V \ntnsp= {\tableau[scY]{5,4|3,1,2}}\,, \\
  \phantom{\sum^\sum} \nTksp
  W = {\tableau[scY]{5,4,1,2,3}}\,.
\end{gathered}
\end{equation}

The statistics we apply to $P$-tableaux are
$P$-analogs of traditional
permutation statistics.
Given a $P$-tableau $U$,
let $U_i$ be the $i$th row (from the bottom) of $U$,
and let $U_{i,j}$ be the $j$th entry in row $i$.
Call a position $(i,j)$ in $U$ a {\em descent}
if $U_{i,j} >_P U_{i,j+1}$,
and a {\em record}
if
$U_{i,1}, \dotsc, U_{i,j-1} <_P U_{i,j}$.
Define
$\des_P(U)$ and $\rec_P(U)$
to be the numbers of descents and records in $U$, respectively,
and call $U$
\begin{enumerate}
\item {\em descent-free} or {\em row-semistrict} if $\des_P(U) = 0$,
\item {\em column-strict} if the entries of each column satisfy
  $U_{i,j} <_P U_{i+1,j}$,
\item {\em standard} if it is column-strict and row-semistrict.
\end{enumerate}
For example,
the tableaux in (\ref{eq:tableauxposet}) satisfy
$\des_P(T) = \des_P(U) 
= 0$,
$\des_P(V) = \des_P(W) = 1$,
and
$\rec_P(W) = 1$,
$\rec_P(U) = \rec_P(V)
= 2$,
$\rec_P(T) = 5$.
Tableaux $T$, $U$, are row-semistrict,
$U$, $V$, $W$ are column-strict, and
$U$ is standard.


In terms of the above definitions, we have the following
combinatorial interpretations of trace evaluations.
Interpretations of induced sign character evaluations are
the simplest.  (See, e.g., \cite[Eqn.~(3.11)]{SkanCCS}.)
\begin{thm}\label{t:epsilonevals}
  Let $G$ be any (simple) graph on $n$ vertices,
  and let $P$ be any poset on $n$ elements.  For all $\lambda \vdash n$
  we have
  \begin{enumerate}
  \item $\epsilon^\lambda(G)$ is $c(G,\lambda)$, the number of
    proper colorings of $G$ of type $\lambda$.
  \item $\epsilon^\lambda(\inc(P))$ is the number of column-strict
    $P$-tableaux of shape $\lambda^\tr$.
  \end{enumerate}
\end{thm}
\noindent
Induced trivial character evaluations are also rather
simple~\cite[Thm.\,4.7 (ii-b)]{CHSSkanEKL}, \cite[Thm.\,3.7]{SkanCCS}.
\begin{thm}\label{t:etaevals}
  Let $G$ be any (simple) graph on $n$ vertices,
  and let $P$ be any poset on $n$ elements.
  For all $\lambda  = (\lambda_1,\dotsc, \lambda_r) \vdash n$
  we have
  \begin{enumerate}
  \item $\eta^\lambda(G)$ is the number of sequences $(O_1,\dotsc,O_r)$
    of acyclic orientations of
    induced subgraphs
    of $G$ on pairwise disjoint vertex
    subsets of cardinalities $\lambda_1,\dotsc,\lambda_r$.
  \item $\eta^\lambda(\inc(P))$ is the number of row-semistrict
    $P$-tableaux of shape $\lambda$.
  \end{enumerate}
\end{thm}
\noindent 
While $\chi^\lambda(G)$ is negative for some graphs $G$,
even for some incomparability graphs $\inc(P)$,
Stanley and Stembridge~\cite[Conj.\,5.1]{StanStemIJT}
conjectured $\chi^\lambda(\inc(P))$ to be nonnegative
for $(\mathbf3 + \mathbf 1)$-free posets $P$.
Gasharov~\cite{GashInc} proved this,
and Kaliszewski~\cite[Prop.\ 4.3]{KaliHook}
showed that when $\lambda$ is a hook partition,
the evaluation $\chi^\lambda(\inc(P))$
is nonnegative for all posets $P$.
Combining these results into one statement, we have the
following.
\begin{thm}\label{t:KaliGash}
  Let $P$ be an $n$-element poset.
  \begin{enumerate}
  \item For $\lambda \vdash n$ a hook shape,
      $\chi^\lambda(\inc(P))$ is the number of
      standard $P$-tableaux of shape $\lambda$.
  \item 
      If $P$ is $(\mathbf3+\mathbf1$)-free then for any
      $\lambda \vdash n$,
      $\chi^\lambda(\inc(P))$ is the number of
      standard $P$-tableaux of shape $\lambda$.
  \end{enumerate}
  \end{thm}

Observation~\ref{o:hookchie} and Theorem~\ref{t:KaliGash}
imply
that when $P$ is an $n$-element chain, the hook irreducible
character evaluations
$\chi^{k1^{n-k}}(\inc(P)) = \tbinom{n-1}{k-1}$
and $\chi^{(k-1)1^{n-k+1}}(\inc(P)) = \tbinom{n-1}{k-2}$
are related by
\begin{equation}\label{eq:hookchareqs}
  (k-1) \chi^{k1^{n-k}}(\inc(P)) = (n-k+1) \chi^{k-1,1^{n-k+1}}(\inc(P))
\end{equation}
for $k = 2,\dotsc,n$.
For an arbitrary poset $P$ no analogous equality
exists, but we do have similar
inequalities.
\begin{lem}\label{l:hookineq}
  For each naturally labeled $n$-element poset $P$ and $k = 2,\dotsc,n$
  we have
  \begin{equation}\label{eq:binomapprox}
    (k-1) \chi^{k1^{n-k}}(\inc(P)) \geq (n-k+1) \chi^{k-1,1^{n-k+1}}(\inc(P)).
  \end{equation}
  Furthermore, the difference
   $(k-1) \chi^{k1^{n-k}}(\inc(P)) - (n-k+1) \chi^{k-1,1^{n-k+1}}(\inc(P))$
  equals the number of standard $P$-tableaux
  of shape $k1^{n-k}$ in which one entry of the first row is marked
  and is incomparable in $P$ to an entry in an earlier column of the tableau.
\end{lem}
\begin{proof}
  Let $\mathcal C_k = \mathcal C_k(P)$
  be the set of standard $P$-tableaux of shape
  $k 1^{n-k}$ in which one entry of column $1$
  other than that in position $(1,1)$ is marked.
  Let $\mathcal R_k = \mathcal R_k(P)$
  be the set of standard $P$-tableaux of shape
  $k 1^{n-k}$ in which one entry of row $1$
  other than that in position $(1,1)$ is marked.
  The inequality (\ref{eq:binomapprox}) asserts that
  $|\mathcal R_k| \geq |\mathcal C_{k-1}|$.
  To see this, we define a family of maps
  $\{f_k: \mathcal C_{k-1} \rightarrow \mathcal R_k \,|\, k = 2,\dotsc, n\}$,
  by letting
  $f_k(U)$ be the tableau constructed from $U$ by
  removing the marked element from the first column of $U$,
  and reinserting it as far as possible to the right in row $1$
  subject to the requirement that it be a record.

  To see that the map $f_k$ is well-defined,
  let $m \in [n]$ be the marked element in column $1$ of $U$.
  Since $U$ is standard,
  $m$ must be greater than the element in position $(1,1)$.
  Thus it will certainly be inserted into positions $(1,2), \dotsc, (1,k-1)$
  or at the end of row $1$, creating a tableau of shape $k1^{n-k}$ with
  one marked element in the first row.
  We claim that the map $f_k$ is injective.
  To see this,
  consider a tableau $V \in \mathcal R_{k}$
  which satisfies $V = f_k(U)$ for some $U \in \mathcal C_{k-1}$.
  A marked element $i$ in row $1$ of $V$ will necessarily be greater in $P$
  than all elements to its left in row $1$, and will be comparable in $P$
  to all elements in column $1$.
  To recover $U$, remove $i$ from row $1$ of $V$
  and insert it into the unique position of column $1$ so that entries there
  increase.  The resulting tableau will still be $P$-semistrict
  because if the element preceding $i$ in $V$, say $h$,
  and the element following $i$ in $V$, say $j$,
  satisfy $h >_P j$, then we have $i >_P h >_P j$, contradicting
  the semistrictness of row $1$ of $V$.
  
  To see that the difference
  $(k-1) \chi^{k1^{n-k}}(P) - (n-k+1) \chi^{k-1,1^{n-k+1}}(P)$
  has the claimed interpretation, consider the
  elements of $\mathcal R_k \ssm f_k(\mathcal C_{k-1})$.
  These are standard $P$-tableaux $V$
  which contain a marked element $i$ in row $1$
  which is not a record in row $1$, or which is incomparable to
  some element of column $1$.  We claim that if $i$ is
  not a record in row $1$, then it must
  be incomparable to an element to its left in row $1$.
  Assume otherwise: assume that $i$ follows $h$ in row $1$ with $h <_P i$
  and that some element $j >_P i$ appears earlier than $h$ in row $1$.
  Assume that $j$ is the rightmost such element with these properties.
  Then we have $j >_P h$.
  Let $g$ be the element immediately following $j$ in row $1$.
  By our choice of $j$ we cannot have $g >_P j$,
  and since $V$ is standard we cannot have $j >_P g$.
  Thus $j$ must be incomparable to $g$.
  By our choice of $j$ we cannot have $g >_P i$,
  and our assumption on $i$ we cannot have $g$ incomparable to $i$.
  Thus we have $g <_P i$.  But then we have $g <_P i <_P j$, a contradiction.
\end{proof}
For example, define $\mathcal C_k$ and $\mathcal R_k$ as in the proof
of \ref{l:hookineq}, and consider
the poset $P$ in (\ref{eq:tableauxposet}).
The set $\mathcal C_3$
consists of two tableaux.  Applying $f_4$ to these we have
\begin{equation*}
  f_4\ntksp\left( \;\tableau[scY]{5|\circd3|1,2,4}\; \right) =
  \;\tableau[scY]{5|1,\circd3,2,4}\,, \qquad
  f_4\ntksp\left( \;\tableau[scY]{\circd 5|3|1,2,4}\; \right) =
  \;\tableau[scY]{3|1,2,\circd5,4}\,,
\end{equation*}
and the difference
$3 \chi^{41}(P) - 2 \chi^{311}(P)$ counts elements of $\mathcal R_4$
whose first row contains a marked element which is not a record
in that row, or which is incomparable to some element of the first column,
\begin{equation*}
  \;\tableau[scY]{5|1,3,\circd2,4}\,, \qquad
  \;\tableau[scY]{5|1,3,2,\circd4}\,, \qquad
  \;\tableau[scY]{5|1,\circd4,3,2}\,, \qquad
  \;\tableau[scY]{5|1,4,\circd3,2}\,, \qquad
  \;\tableau[scY]{5|1,4,3,\circd2}\,, \dotsc.
  \end{equation*}

The statement preceding Theorem~\ref{t:KaliGash} implies
that $\phi^\lambda(G)$ is negative for some graphs $G$,
even for some incomparability graphs $G = \inc(P)$.
Hikita~\cite[Thm.\,3]{HikitaProof} showed that it is positive
for $G = \inc(P)$ and $P$ a $(\mathbf3 + \mathbf1)$-free poset,
settling the Stanley--Stembridge conjecture~\cite[Conj.\,5.5]{StanStemIJT}.
\begin{thm}\label{t:hikita}
  For $P$ a $(\mathbf3 + \mathbf1)$-free poset
  and $\lambda \vdash n$,
  we have $\phi^\lambda(\inc(P)) \geq 0$.
\end{thm}

\noindent
Hikita's proof unfortunately
does not provide a combinatorial interpretation
of $\phi^\lambda(\inc(P))$.
\bp
Find a combinatorial interpretation for 
$\phi^\lambda(\inc(P))$ which holds for all $\lambda \vdash n$ and for
$n$-element posets $P$ avoiding $\mathbf3 + \mathbf1$.
\ep

A related result for
monomial traces~\cite[Lem.\,4.1]{AthanPSE}, \cite[Thm.\,3.3]{StanSymm}
concerns sums
of the form
  \begin{equation}\label{eq:thetaldef}
    \theta^\ell = \sumsb{\mu \vdash n\\\ell(\mu) = \ell} \ntnsp \phi^\mu.
  \end{equation}
\begin{prop}\label{p:stanksources}
  Let $G$ be any (simple) graph on $n$ vertices,
  and let $P$ be any poset on $n$ elements.
  The traces $\{ \theta^\ell \,|\, 1 \leq \ell \leq n \}$
  satisfy
  \begin{enumerate}
  \item $\theta^\ell(G)$ is the number of acyclic
    orientations of $G$ having $\ell$ sources,
    \item
      $\theta^\ell(\inc(P))$ is the number of
      descent-free $P$-tableaux of shape $n$ having $\ell$
      records.
  \end{enumerate}
\end{prop}



\section{Applications to total nonnegativity}\label{s:tnn}

Nonnegative expansions of chromatic symmetric functions in
the standard bases
are closely
related to
the immanants defined in (\ref{eq:immdef})
and to certain directed planar graphs.
We will make these relationships precise in Proposition~\ref{p:immincP}
and state some immanantal analogs of results from Section~\ref{s:chrom}.

Define a (nonnegative weighted)
{\em planar network of order $n$} to be a directed, planar, acyclic
digraph $F = (V,E)$
which can be embedded in a disc so that $2n$ distinguished vertices
labeled clockwise as $s_1,\dotsc,s_n,t_n,\dotsc,t_1$
lie on the boundary of the disc,
with a nonnegative real {\em weight} $c_{u,v}$
assigned to each edge $(u,v) \in E$.
We may assume that $s_1,\dotsc,s_n$, called {\em sources}, have indegree $0$
and that $t_n,\dotsc,t_1$, called {\em sinks}, have outdegree $0$.
To every source-to-sink path, we associate a weight equal to the product
of its edge weights,
and we define the {\em path matrix}
$A = A(F) = (a_{i,j})_{i,j \in [n]}$ by
setting $a_{i,j}$ equal to the sum of weights of all paths
from $s_i$ to $t_j$.
For example, one
planar network $F$ of order $3$, with edges weighted by positive numbers
$1, a, \dotsc, h$,
and its path matrix $A$ are

\begin{equation}\label{eq:planarnet}
    F = \ntnsp
\begin{tikzpicture}[scale=.6,baseline=-5]
\draw[fill] (0,2) circle (1mm); \node at (-.5,2) {$s_3$};
\draw[fill] (0,0) circle (1mm); \node at (-.5,0) {$s_2$};
\draw[fill] (0,-2) circle (1mm); \node at (-.5,-2) {$s_1$};
\draw[fill] (3,2) circle (1mm); 
\draw[fill] (2,0) circle (1mm); 
\draw[fill] (2,-2) circle (1mm); 
\draw[fill] (4,0) circle (1mm); 
\draw[fill] (6,2) circle (1mm); \node at (6.5,2) {$t_3$};
\draw[fill] (6,0) circle (1mm); \node at (6.5,0) {$t_2$};
\draw[fill] (6,-2) circle (1mm); \node at (6.5,-2) {$t_1$};
\node at (.8,.35) {$1$};
\node at (.8,-1.65) {$1$};
\node at (1.2,1.25) {$1$};
\node at (1.2,-.75) {$a$};
\node at (1.4,2.35) {$d$};
\node at (2.25,1.25) {$e$};
\node at (2.7,-.75) {$b$};
\node at (3,.35) {$f$};
\node at (3.7,1.25) {$1$};
\node at (4,-1.65) {$1$};
\node at (4.5,2.35) {$g$};
\node at (4.5,-1.15) {$c$};
\node at (4.8,1.25) {$h$};
\node at (5,.35) {$1$};
\draw[->,thick] (0,2) -- (2.85,2);
\draw[->,thick] (0,2) -- (1.88,0.07);
\draw[->,thick] (0,0) -- (1.85,0);
\draw[->,thick] (0,0) -- (1.88,-1.93);
\draw[->,thick] (0,-2) -- (1.85,-2);
\draw[->,thick] (2,0) -- (2.92,1.9);
\draw[->,thick] (2,0) -- (3.85,0);
\draw[->,thick] (2,-2) -- (3.88,-.07);
\draw[->,thick] (2,-2) -- (5.88,-.07);
\draw[->,thick] (2,-2) -- (5.85,-2);
\draw[->,thick] (3,2) -- (5.85,2);
\draw[->,thick] (3,2) -- (3.9,.1);
\draw[->,thick] (4,0) -- (5.88,1.93);
\draw[->,thick] (4,0) -- (5.85,0);
\end{tikzpicture}\ntnsp,
\qquad
A = \begin{bmatrix}
  1 & b+c & \ntnsp bh \\
  a & ab\ntnsp+{\ntnsp}ac\ntnsp+{\ntnsp}e\ntnsp+{\ntnsp}f & \ntnsp abh+fh+eh+eg \\
  0 & d+e+f & \ntnsp dh\ntnsp+\ntnsp eh\ntnsp+\ntnsp fh\ntnsp+\ntnsp eg\ntnsp+\ntnsp dg
\end{bmatrix}\ntksp.\nTksp
\end{equation}

A result often attributed to Lindstr\"om
~\cite{LinVrep} but proved earlier by Karlin and McGregor~\cite{KMG}
asserts the total nonnegativity of such a matrix.
\begin{thm}\label{t:lin}
  The path matrix $A$ of a nonnegative weighted planar network $F$ of order $n$
  is totally nonnegative.  Moreover,
  the nonnegative number $\det(A)$ equals
  \begin{equation*}
    \sum_\pi \wgt(\pi),
  \end{equation*}
  where the sum is over all families $\pi = (\pi_1,\dotsc,\pi_n)$
  of pairwise nonintersecting paths in $F$,
  with $\pi_i$ a path from $s_i$ to $t_i$ for $i = 1,\dotsc,n$, and where
  \begin{equation}\label{eq:wgt}
    \wgt(\pi) \defeq \wgt(\pi_1) \cdots \wgt(\pi_n).
    \end{equation}
\end{thm}

\noindent
Thus by inspection of the network $F$ in (\ref{eq:planarnet}),
its path matrix $A$ satisfies $\det(A) = fdg$.

The converse of Theorem~\ref{t:lin} is true as well.
That is,
path matrices are essentially the only examples of totally nonnegative
matrices
~\cite{BrentiCTP}, \cite{CryerProp}, \cite{loewner}, \cite{WReduction}.
%
\begin{thm}\label{t:tnnconv}
  For each $n \times n$ totally nonnegative matrix $A$, there exists
  a nonnegative weighted planar network of order $n$
  whose path matrix is $A$.
\end{thm}

Also belonging to the subject of total nonnegativity are polynomial functions
\begin{equation*}
  f(x)
  \defeq
  f(x_{1,1}, x_{1,2}, \dotsc, x_{n,n}) \in
  \mathbb Z[x_{1,1},x_{1,2},\dotsc,x_{n,n}]
  \end{equation*}
having the property that
\begin{equation*}
  f(A) \defeq f(a_{1,1},a_{1,2},\dotsc,a_{n,n}) \geq 0
\end{equation*}
for all totally nonnegative matrices $A = (a_{i,j})$.
Interest in such polynomials comes from the fact that
elements of a certain {\em dual canonical basis} of
$\mathbb Z[x_{1,1}, x_{1,2}, \dotsc, x_{n,n}]$ have this property~\cite{LusztigTP}.
(See also \cite{RSkanKLImm}.)
Certainly subtraction-free polynomials such as
$\imm{\psi^\lambda}(x)$
are totally nonnegative.
(See (\ref{eq:psidef}).)
Sums of products of minors such as $\imm{\epsilon^\lambda}(x)$ in (\ref{eq:lmw})
are as well, as are the analogous sums of products of permanents
~\cite{LittlewoodTGC}, \cite{MerWatIneq}
\begin{equation}\label{eq:lmw2}
  \imm{\eta^\lambda}(A) = \sum \perm(A_{I_1,I_1}) \cdots \perm(A_{I_\ell,I_\ell}).
\end{equation}
The total nonnegativity of other polynomials is less obvious.
For instance, Stembridge showed that all character immanants are
totally nonnegative~\cite[Cor.\,3.3]{StemImm}.
\begin{thm}\label{t:stemimm}
  For $\lambda \vdash n$
  the polynomial $\imm{\chi^\lambda}(x)$ is totally nonnegative.
\end{thm}

For some totally nonnegative polynomials $\imm\theta(x)$,
one can combinatorially interpret the
evaluation
$\imm\theta(A)$ when $A$ is a totally nonnegative matrix.
Such an interpretation typically employs
a planar network $F$ having path matrix $A$,
guaranteed to exist by Theorem~\ref{t:tnnconv},
and families of paths in $F$
from all sources
to all sinks.
In particular, for a multiset $K$ of edges of $F$, let
$\Pi_e(K)$ denote
the set of all
path families
$\pi = (\pi_1,\dotsc,\pi_n)$
with $\pi_i$ a path from source $i$ to sink $i$,
whose multiset of edges is $K$.
Call $K$ a {\em bijective skeleton in $F$} if $\Pi_e(K)$ is nonempty.
Define $\wgt(K)$ to be the product of weights of edges in $K$,
with multiplicities, so that $\wgt(\pi) = \wgt(K)$ for all $\pi \in \Pi_e(K)$.
For each path family $\pi \in \Pi_e(K)$, define the poset
$P = P(\pi)$ by declaring $\pi_i < \pi_j$ if
$i < j$ as integers
and
$\pi_i$ does not intersect $\pi_j$.
We will refer to the union of $P(\pi)$-tableaux, over all
path families $\pi$ covering a bijective skeleton of $F$,
\begin{equation}\label{eq:Ftableaux}
  \{ \text{$U$ a $P(\pi)$-tableau} \,|\,
  \pi \in \Pi_e(K), \text{ $K$ a bijective skeleton in $F$} \}
\end{equation}
as the set of {\em $F$-tableaux}.  These are fillings of Young diagrams
with path families in $F$.
For $F$-tableau $U$ containing path family $\pi \in \Pi_e(K)$,
we define $\wgt(U) \defeq \wgt(K)$.

For example, consider the network $F$ in (\ref{eq:planarnet})
and three multisets of edges
  \begin{equation}\label{eq:3multisets}
  K_1 =
\begin{tikzpicture}[scale=.7,baseline=-5]
\draw[gray,-,thin,densely dotted] (0,1) -- (3,1);
\draw[-,ultra thick] (0,0) -- (2,0) -- (3,1);
\draw[-,ultra thick] (0,-1) -- (3,-1);
\draw[gray,-,thin,densely dotted] (0,0) -- (1,-1) -- (3,0);
\draw[gray,-,thin,densely dotted] (1,-1) -- (3,1);
\draw[-,ultra thick] (0,1) -- (1,0) -- (1.5,1) -- (2,0) -- (3,0);
\end{tikzpicture}\,,
\qquad
K_2 
=
\begin{tikzpicture}[scale=.7,baseline=-5]
\draw[gray,-,thin,densely dotted] (0,1) -- (3,1);
\draw[-,ultra thick] (1,0) -- (2,0) -- (3,1);
\draw[-,ultra thick] (0,0) -- (1,-1) -- (2,0);
\draw[-,ultra thick] (0,-1) -- (3,-1);
\draw[gray,-,thin,densely dotted] (0,0) -- (1,-1) -- (3,0);
\draw[gray,-,thin,densely dotted] (1,-1) -- (3,1);
\draw[-,ultra thick] (0,1) -- (1,0) -- (3,0);
\draw[gray,-,thin,densely dotted] (1,0) -- (1.5,1) -- (2,0);
\node at (1.5,-.8) {$\phantom{(2)}$};
\node at (1.5,.8) {$\phantom{(2)}$};  
\end{tikzpicture}\,,
\qquad
K_3
=
\begin{tikzpicture}[scale=.7,baseline=-5]
\draw[gray,-,thin,densely dotted] (0,1) -- (3,1);
\draw[-,ultra thick] (0,0) -- (2,0) -- (3,1);
\draw[-,ultra thick] (0,-1) -- (3,-1);
\draw[gray,-,thin,densely dotted] (0,0) -- (1,-1) -- (3,0);
\draw[gray,-,thin,densely dotted] (1,-1) -- (3,1);
\draw[-,ultra thick] (0,1) -- (1,0) -- (3,0);
\draw[gray,-,thin,densely dotted] (1,0) -- (1.5,1) -- (2,0);
\node at (1.5,.3) {$_{(2)}$};  
\end{tikzpicture}\,,
\end{equation}
  where the marked edge in $K_3$ has
  multiplicity $2$.
  The multisets have weights
  $\wgt(K_1) = feh$,
  $\wgt(K_2) = abfh$,
  $\wgt(K_3) = f^2h$.
The path families $\pi$, $\rho$, $\sigma$, $\tau$ defined by
\begin{equation}\label{eq:4pathfams}
\begin{tikzpicture}[scale=.7,baseline=-5]
\draw[gray,-,thin,densely dotted] (0,1) -- (3,1);
\draw[-,ultra thick] (0,-.0) -- (2,-.0) -- (3,-.0);
\draw[-,thick,dashed] (0,-1) -- (3,-1);
\draw[gray,-,thin,densely dotted] (0,0) -- (1,-1) -- (3,0);
\draw[gray,-,thin,densely dotted] (1,-1) -- (3,1);
\draw[-,ultra thick, dotted] (0,1) -- (1,0) -- (1.5,1) -- (2,0) -- (3,1);
  \node at (-.5,1) {$\pi_3$};  
  \node at (-.5,0) {$\pi_2$};  
  \node at (-.5,-1) {$\pi_1$};  
\end{tikzpicture}\,,
\qquad
\begin{tikzpicture}[scale=.7,baseline=-5]
\draw[gray,-,thin,densely dotted] (0,1) -- (3,1);
\draw[-,ultra thick] (0,-.0) -- (1,-0.0) -- (1.5,1) -- (2,-0.0) -- (3,-.0);
\draw[-,thick,dashed] (0,-1) -- (3,-1);
\draw[gray,-,thin,densely dotted] (0,0) -- (1,-1) -- (3,0);
\draw[gray,-,thin,densely dotted] (1,-1) -- (3,1);
\draw[-,ultra thick, dotted] (0,1) -- (1,0.0) -- (2,0.0) -- (3,1);
  \node at (-.5,1) {$\rho_3$};  
  \node at (-.5,0) {$\rho_2$};  
  \node at (-.5,-1) {$\rho_1$};  
\end{tikzpicture}\,,
\qquad
\begin{tikzpicture}[scale=.7,baseline=-5]
\draw[gray,-,thin,densely dotted] (0,1) -- (3,1);
\draw[gray,-,thin,densely dotted] (0,0) -- (1,0);
\draw[-,thick,dashed] (0,-1) -- (3,-1);
\draw[gray,-,thin,densely dotted] (0,0) -- (1,-1) -- (3,0);
\draw[-,ultra thick] (0,0) -- (1,-1) -- (2,0) -- (3,0);
\draw[-,ultra thick, dotted] (0,1) -- (1,0) -- (2,0) -- (3,1);
\draw[gray,-,thin,densely dotted] (1,0) -- (1.5,1) -- (2,0);
  \node at (-.5,1) {$\sigma_3$};  
  \node at (-.5,0) {$\sigma_2$};  
  \node at (-.5,-1) {$\sigma_1$};  
\end{tikzpicture}\,,
\qquad
\begin{tikzpicture}[scale=.7,baseline=-5]
\draw[gray,-,thin,densely dotted] (0,1) -- (3,1);
\draw[-,ultra thick] (0,-.035) -- (2,-.035) -- (3,-0.035);
\draw[-,thick,dashed] (0,-1) -- (3,-1);
\draw[gray,-,thin,densely dotted] (0,0) -- (1,-1) -- (3,0);
\draw[gray,-,thin,densely dotted] (1,-1) -- (3,1);
\draw[-,ultra thick, dotted] (0,1) -- (1,0.035) -- (2,0.035) -- (3,1);
\draw[gray,-,thin,densely dotted] (1,0) -- (1.5,1) -- (2,0);
  \node at (-.5,1) {$\tau_3$};  
  \node at (-.5,0) {$\tau_2$};  
  \node at (-.5,-1) {$\tau_1$};  
\end{tikzpicture}
\end{equation}
satisfy $\Pi_e(K_1) = \{ \pi, \rho \}$,
$\Pi_e(K_2) = \{ \sigma \}$,
$\Pi_e(K_3) = \{ \tau \}$,
and have posets
\begin{equation*}
P(\pi) = \ntksp
\begin{tikzpicture}[scale=.7,baseline=-5]
\draw[fill] (0,.5) circle (1mm); \node at (0,1) {$\pi_2$};
\draw[fill] (.4,-.5) circle (1mm); \node at (.4,-1) {$\pi_1$};
\draw[fill] (.8,.5) circle (1mm); \node at (.8,1) {$\pi_3$};
\draw[-,thick] (0,.5) -- (.4,-.5);
\draw[-,thick] (.8,.5) -- (.4,-.5);
\end{tikzpicture}\ntksp,
\qquad
P(\rho) = \ntksp
\begin{tikzpicture}[scale=.7,baseline=-5]
\draw[fill] (0,.5) circle (1mm); \node at (0,1) {$\rho_2$};
\draw[fill] (.4,-.5) circle (1mm); \node at (.4,-1) {$\rho_1$};
\draw[fill] (.8,.5) circle (1mm); \node at (.8,1) {$\rho_3$};
\draw[-,thick] (0,.5) -- (.4,-.5);
\draw[-,thick] (.8,.5) -- (.4,-.5);
\end{tikzpicture}\ntksp,
\qquad
P(\sigma) = \ntnsp
\begin{tikzpicture}[scale=.7,baseline=-5]
\draw[fill] (.8,.5) circle (1mm); \node at (.8,1) {$\sigma_3$};
\draw[fill] (.8,-.5) circle (1mm); \node at (.8,-1) {$\sigma_1$};
\draw[fill] (1.6,0) circle (1mm); \node at (1.6,-.5) {$\sigma_2$};
\draw[-,thick] (.8,.5) -- (.8,-.5);
\end{tikzpicture}\ntksp,
\qquad 
P(\tau) = \ntksp
\begin{tikzpicture}[scale=.7,baseline=-5]
\draw[fill] (0,.5) circle (1mm); \node at (0,1) {$\tau_2$};
\draw[fill] (.4,-.5) circle (1mm); \node at (.4,-1) {$\tau_1$};
\draw[fill] (.8,.5) circle (1mm); \node at (.8,1) {$\tau_3$};
\draw[-,thick] (0,.5) -- (.4,-.5);
\draw[-,thick] (.8,.5) -- (.4,-.5);
\end{tikzpicture}\ntksp.
\end{equation*}
The standard $F$-tableaux
\begin{equation*}
  \tableau[scY]{\pi_3|\pi_1,\pi_2}\,,\qquad
  \tableau[scY]{\rho_3|\rho_1,\rho_2}\,,\qquad
  \tableau[scY]{\sigma_3|\sigma_1,\sigma_2}\,,\qquad
  \tableau[scY]{\tau_3|\tau_1,\tau_2}  
\end{equation*}
have weights $feh$, $feh$, $abfh$, $f^2h$, respectively.

For an $n \times n$ totally nonnegative matrix $A$
and a trace $\theta \in \trspace n$, 
we may compute $\imm \theta(A)$ by considering
a planar network $F$ having path matrix $A$,
the union over bijective skeletons $K$ of
path families $\pi \in \Pi_e(K)$,
and the corresponding
chromatic symmetric functions $X_{\inc(P(\pi))}$.
Specifically, we have the following~\cite[Cor.\,4.6]{SkanCCS}.
\begin{prop}\label{p:immincP}
  For $F$ a planar network having path matrix $A$ and $\theta \in \trspace n$,
  we have
\begin{equation}\label{eq:propcorcomb}
  \imm{\theta}(A) = \sum_K \wgt(K)
  \nTksp \ntnsp \sum_{\pi \in \Pi_e(K)} \nTksp \theta(\inc(P(\pi))),
\end{equation}
where $K$ varies over all bijective skeletons in $F$.  If
for all posets $P$, $\theta(\inc(P))$ counts $P$-tableaux
having a particular property,
then we have
\begin{equation}\label{eq:propcorcomb2}
  \imm{\theta}(A) = \sum_U \wgt(U),
\end{equation}
where the sum is over $F$-tableaux $U$ having the property.
\end{prop}
Proposition~\ref{p:immincP}
has the following consequence~\cite[Cor.\,4.7]{SkanCCS}.
\begin{cor}\label{c:immincP}
  If $\theta \in \trspace n$ satisfies
  $\theta(\inc(P)) \geq 0$ for all posets $P$, then
the polynomial $\imm\theta(x)$ is totally nonnegative.
\end{cor}

By Theorems~\ref{t:epsilonevals} -- \ref{t:KaliGash},
Corollary~\ref{c:immincP} applies to
induced sign character immanants,
induced trivial character immanants,
and irreducible character immanants indexed by hook partitions.
It does not apply to
irreducible character immanants in general, because we have
$\chi^\lambda(\inc(P)) < 0$ for some $\lambda$, $P$.
\begin{thm}\label{t:hookimm}
  For $k \leq n$,
  the polynomial $\imm{\chi^{k1^{n-k}}}(x)$ is totally nonnegative.
    In particular, for $A$ the path matrix of planar network $F$, we have
    \begin{equation*}
      \imm{\chi^{k1^{n-k}}}(A) = \sum_U \wgt(U),
    \end{equation*}
    where the sum is over all standard $F$-tableaux of shape $k1^{n-k}$.
\end{thm}
\bp\label{p:stemimm}
Combinatorially interpret the numbers $\imm{\chi^\lambda}(A)$ in
Theorem~\ref{t:stemimm}.
\ep

Corollary~\ref{c:immincP} also does not apply to monomial traces,
which satisfy $\phi^\lambda(\inc(P)) < 0$ for some $\lambda$, $P$.
Nevertheless, Stembridge conjectured that monomial trace immanants are totally
nonnegative~\cite[Conj.\,2.1]{StemConj}.
\begin{conj}\label{c:monimm}
  For $\lambda \vdash n$ the polynomial $\imm{\phi^\lambda}(x)$ is totally
  nonnegative.
\end{conj}


Some evidence for Conjecture~\ref{c:monimm} follows from
Hikita's result in Theorem~\ref{t:hikita}.
\begin{prop}
  If $A$ is the antiadjacency matrix of a unit interval order
  labeled as in Step 2 of Algorithm~\ref{a:PtoC},
  then for all $\lambda \vdash n$ we have that $\imm{\phi^\lambda}(A) \geq 0$.
\end{prop}
\begin{proof}
  Let $A$ be the antiadjacency matrix of unit interval order $P$,
  labeled as in Step 2 of Algorithm~\ref{a:PtoC},
  and let $w$ be the $312$-avoiding permutation
  associated to $P$ by Step 3 of Algorithm~\ref{a:PtoC}.
  It is well known that the entries of $A$ which are equal to $1$
  form a Ferrers shape, and that we have
  \begin{equation*}
    \permmon av = \begin{cases} 1 &\text{if $v \leq w$},\\
      0 &\text{otherwise}.
    \end{cases}
  \end{equation*}
  (See, e.g., \cite[Prop.\,19, Prop.\,22]{WatsonBruhat}.)  Thus we have
  \begin{equation*}
    \imm{\phi^\lambda}(A) =
    \sum_{v \in \sn} \phi^\lambda(v) \permmon av
    = \sum_{v \leq w} \phi^\lambda(v) 
    = \phi^\lambda \Big( \sum_{v \leq w} v \Big)
    = \phi^\lambda(C'_w(1)).
  \end{equation*}
  By Proposition~\ref{p:uioto312avoid} this number is $\phi^\lambda(\inc(P))$,
  and by Theorem~\ref{t:hikita} it is nonnegative.
\end{proof}

More evidence for Conjecture~\ref{c:monimm} follows from
Athanasiadis and Stanley's results in
Proposition~\ref{p:stanksources}.
\begin{prop}\label{p:sumofmonimms}
  For $\ell = 1,\dotsc,n$, the sum
    \begin{equation}\label{eq:sumofmux}
        \sumsb{\mu \vdash n\\\ell(\mu) = \ell} \ntnsp \imm{\phi^\mu}(x)
    \end{equation}
    of monomial immanants is a totally nonnegative polynomial.
    In particular, for $A$ the path matrix of planar network $F$, we have
    \begin{equation}\label{eq:sumofmuA}
      \sumsb{\mu \vdash n\\\ell(\mu) = \ell} \ntnsp \imm{\phi^\mu}(A) =
      \sum_U \wgt(U), 
    \end{equation}
    where the sum is over row-semistrict $F$-tableaux $U$
    of shape $n$ having $\ell$ records.
\end{prop}
\begin{proof}
  Fix a totally nonnegative matrix $A$
  and
 define the trace
  \begin{equation}
    \theta^\ell = \sumsb{\mu \vdash n\\\ell(\mu) = \ell} \ntnsp \phi^\mu
  \end{equation}
  so that the left-hand side of (\ref{eq:sumofmuA}) is $\imm{\theta^\ell}(A)$.
  By
  Proposition~\ref{p:immincP}
  we have that
  \begin{equation*}
    \imm{\theta^\ell}(A) =
    \sum_K \wgt(K) \nTksp
    \sum_{\pi \in \Pi_e(K)} \nTksp \theta^\ell(\inc(P(\pi))),
  \end{equation*}
  where the first sum is over all bijective skeletons $K$ in $F$.
  By Proposition~\ref{p:stanksources}, the inner sum equals the
  number of
  descent-free $P(\pi)$-tableaux of shape $n$
  having $\ell$
  records.
  Equivalently, it equals the number of
  row-semistrict $F$-tableaux of shape $n$ having $\ell$
  records.
  Thus the evaluation (\ref{eq:sumofmuA}) has the desired interpretation.
  Since each such $F$-tableau has weight
  $\wgt(K) \geq 0$, 
  the polynomial (\ref{eq:sumofmux}) is totally nonnegative.  
\end{proof}
For example, let $F$, $A$ be as in (\ref{eq:planarnet}).
The path families $\pi$, $\rho$, $\sigma$, and $\tau$
in (\ref{eq:4pathfams})
lead to $7$ tableaux
having $2$ records (in positions $1$ and $2$):
\begin{equation*}
\tableau[scY]{\pi_1,\pi_2,\pi_3}\,, \quad
\tableau[scY]{\pi_1,\pi_3,\pi_2}\,, \quad
\tableau[scY]{\rho_1,\rho_2,\rho_3}\,, \quad
\tableau[scY]{\rho_1,\rho_3,\rho_2}\,, \quad
\tableau[scY]{\sigma_1,\sigma_3,\sigma_2}\,, \quad
\tableau[scY]{\tau_1,\tau_2,\tau_3}\,, \quad
\tableau[scY]{\tau_1,\tau_3,\tau_2}\,.
\end{equation*}
The weights of these tableaux,
as described in \eqref{eq:3multisets} -- \eqref{eq:4pathfams},
lead to the computation
$\imm{\theta^2}(A) = 4feh + abfh + 2f^2h$
in \eqref{eq:sumofmuA}.


\section{Main result and open problems}\label{s:main}

We now show that Heyfron's inequalities (Theorem~\ref{t:heyfron})
hold not only for Hermitian positive semidefinite matrices,
but also for totally nonnegative matrices.

\begin{thm}\label{t:main}
  For each $n \times n$ totally nonnegative matrix $A$ we have
\begin{equation*}
   \perm(A)=\frac{\imm{\chi^n}(A)}{\chi^{n}(e)}\geq \frac{\imm{\chi^{n-1,1}}(A)}{\chi^{n-1,1}(e)}\geq \frac{\imm{\chi^{ n-2,1,1}}(A)}{\chi^{n-2,1,1}(e)}\geq \cdots \geq \frac{\imm{\chi^{1,\dotsc,1}}(A)}{\chi^{1,\dotsc,1}(e)}=\det(A).
\end{equation*} 
Equivalently, for $k = 2,\dotsc,n$,
the difference
        \begin{equation}\label{eq:thediff}
        \frac{\imm{\chi^{k1^{n-k}}}(x)}{\chi^{k1^{n-k}}(e)}
        -
        \frac{\imm{\chi^{(k-1)1^{n-k+1}}}(x)}{\chi^{(k-1)1^{n-k+1}}(e)}
        =
         \frac{\imm{\chi^{k1^{n-k}}}(x)}{\binom{n-1}{k-1}}
        -
        \frac{\imm{\chi^{(k-1)1^{n-k+1}}}(x)}{\binom{n-1}{k-2}}
    \end{equation}
        is a totally nonnegative polynomial.
\end{thm}
\begin{proof}[First proof]
        Multiplying (\ref{eq:thediff}) by $\tfrac{(n-1)!}{(n-k)!(k-2)!}$,
        we have
        \begin{equation}\label{eq:theseconddiff}
          (k-1)\imm{\chi^{k1^{n-k}}}(x) - (n-k+1)\imm{\chi^{(k-1)1^{n-k+1}}}(x).
        \end{equation}
        Let $A$ be an $n \times n$ totally nonnegative matrix.
        By Theorem~\ref{t:tnnconv}, we may choose a planar network $F$
        whose path matrix is $A$.
        By Theorem~\ref{t:hookimm},
        the evaluation of (\ref{eq:theseconddiff}) at $A$ equals
        \begin{equation}\label{eq:eval}
          \begin{aligned}
          &(k-1) \ntnsp \sum_\pi \wgt(\pi) d_k(\pi) -
          (n-k+1) \ntnsp \sum_\pi \wgt(\pi) d_{k-1}(\pi)\\
            &= \sum_\pi \wgt(\pi) [(k-1)d_k(\pi) - (n-k+1)d_{k-1}(\pi)],
          \end{aligned}
        \end{equation}
        where $d_k(\pi)$ is the number of standard $\pi$-tableaux
        of shape $k 1^{n-k}$, and
        the sums are over
        \begin{equation*}
          \pi \in \bigcup_K
          \Pi_e(K)
        \end{equation*}
        with $K$ varying over all bijective skeletons in $F$.
        By Lemma~\ref{l:hookineq}
        the difference in square brackets in the last sum equals
        the number of standard $P(\pi)$-tableaux
        of shape $k1^{n-k}$ with one marked
        entry in columns $2,\dotsc,k$ which is incomparable to at
        least one entry in an earlier column.
        Since this number and $\wgt(\pi)$ are nonnegative, the evaluation
        (\ref{eq:eval}) is nonnegative.  Thus the polynomial
        (\ref{eq:theseconddiff}) is totally nonnegative
        and so is the polynomial (\ref{eq:thediff}).
\end{proof}
\begin{proof}[Second proof]
  Define the traces $\theta^1, \dotsc, \theta^n$
  as in (\ref{eq:thetaldef}).
  By Proposition~\ref{p:hookKostka}
  we have that
  the hook irreducible character immanant indexed by $k1^{n-k}$
  belongs to the $n$-dimensional space spanned by
  $\imm{\theta^1}(x), \dotsc, \imm{\theta^n}(x)$.  Specifically,
    \begin{equation}\label{eq:hookirrmon}
      \imm{\chi^{k1^{n-k}}}(x) =
      \nTksp \sum_{\ell = n-k+1}^n \ntksp \binom{\ell-1}{n-k} \, \imm{\theta^\ell}(x).
    \end{equation}
It follows that for $k = 2,\dotsc,n$, the difference (\ref{eq:thediff})
      expands in the basis of $\theta^\ell$-immanants as
      $c_{k,1} \imm{\theta^1}(x) + \cdots + c_{k,n} \imm{\theta^n}(x)$
    with nonnegative coefficients
    \begin{equation*}
        c_{k,\ell} = \frac{\binom{\ell-1}{n-k}}{\binom{n-1}{k-1}} - \frac{\binom{\ell-1}{n-k+1}}{\binom{n-1}{k-2}} = 
        \begin{cases}
        0
        &\text{if $\ell = 1,\dotsc,n-k$},\\
        \cfrac{\frac{n-\ell}{k-1}\tbinom{\ell-1}{n-k}}{\tbinom{n-1}{k-1}}
        &\text{if $\ell = n-k+1,\dotsc,n$.}
        \end{cases}
    \end{equation*}
    By Proposition~\ref{p:sumofmonimms}, each immanant $\imm{\theta^\ell}(x)$
    is a totally nonnegative polynomial.
    Thus for $k = 2,\dotsc, n$
    the difference (\ref{eq:thediff}) is as well.
\end{proof}
Theorem~\ref{t:main} thus provides some progress on the problem
of understanding (\ref{eq:irrimmineq}).
\bp
Characterize the pairs $(\lambda,\mu)$
of partitions satisfying
\begin{equation}\label{eq:irrimmineq2}
  \frac{\imm{\chi^\lambda}(A)}{\chi^\lambda(e)}
  \geq
  \frac{\imm{\chi^\mu}(A)}{\chi^\mu(e)}
\end{equation}
for all totally nonnegative
or Hermitian positive semidefinite matrices.
\ep

Generalizing Heyfron's inequalities, Pate~\cite{PatePICIGMF}
proved that if $\lambda$, $\mu$
are two partitions of $n$, with $k$ the
multiplicity of $\lambda_1$ in $\lambda = (\lambda_1,\dotsc,\lambda_\ell)$
and
\begin{equation}\label{eq:pate}
  \mu = (\lambda_1 - 1, \dotsc, \lambda_k - 1,
  \lambda_{k+1}, \dotsc, \lambda_\ell, \;
  \underbrace{\ntksp 1, \dotsc, 1 \ntksp }_k\;),
  \end{equation}
then we have (\ref{eq:irrimmineq2}) for all Hermitian positive semidefinite
matrices.  In other words, the Young diagram of $\mu$ is obtained by
removing the rightmost column of the Young diagram of $\lambda$
and by appending this to the first column, for example
\begin{equation*}
  \lambda =
   (4, 4, 3, 2) =
      \tableau[scY]
        {\ , \ |
         \ , \ , \ |
         \ , \ , \ , \ |
         \ , \ , \ , \ }\,,
      \qquad
      \mu =
   (3, 3, 3, 2, 1, 1) =
      \tableau[scY]
        {\ |
         \ |
         \ , \ |
         \ , \ , \ |
         \ , \ , \ |
         \ , \ , \ }\,.
\end{equation*}
Perhaps this inequality holds
for totally nonnegative matrices as well.
\bp
Decide if Pate's inequality, i.e., (\ref{eq:irrimmineq2}) for $\lambda$, $\mu$
satisfying (\ref{eq:pate}), holds for all totally nonnegative matrices.
\ep

Theorem~\ref{t:main} and its proofs suggest several other open problems.
Haiman~\cite[Conj.\.2.1]{HaimanHecke} has conjectured that
certain $q$-analogs $\phi_q^\lambda$ and $C'_w(q)$
of the monomial trace $\phi^\lambda$
and Kazhdan--Lusztig basis element $C'_w(1)$
satisfy
\begin{equation*}
  \phi_q^\lambda(\qp{\ell(w)}2C'_w(q)) \in \mathbb N[q]
  \end{equation*}
for all $\lambda$ and all $w$.
(See \cite{HaimanHecke} for definitions.)
Perhaps the following weaker statement would be easier to prove.
\bp
Show that for $\ell = 1,\dotsc,n$ and all $w \in \sn$ we have
$\theta_q^\ell(\qp{\ell(w)}2C'_w(q)) \in \mathbb N[q]$,
where
\begin{equation*}
  \theta_q^\ell = \sumsb{\lambda \vdash n\\ \ell(\lambda) = \ell} \phi_q^\lambda.
  \end{equation*}
\ep
\noindent
This is known to be true for $w$ \avoidingp.
(See, e.g., \cite[Thm.\,5.6]{CHSSkanEKL}, \cite[Prop.\,5.4]{SkanCCS}.)

It would also be interesting to find an analog of the
Littlewood--Merris--Watkins identities (\ref{eq:lmw}), (\ref{eq:lmw2})
for the immanants
$\{ \imm{\theta^\ell}(x) \,|\, \ell \in [n] \}$.

\bp For $\ell = 1,\dotsc, n$, find an expression for the
immanant $\imm{\theta^\ell}(x)$ which makes its
total nonnegativity apparent.
\ep


\section{Acknowledgements}

The author is grateful to Alex Fink
and Keith McNulty for helpful conversations,
and to
the Mathematics Genealogy Project
for displaying Keith McNulty's relevant but unpublished thesis.

\bibliography{skan}
\end{document}